\documentclass[11pt,a4paper]{article}

\usepackage{epsf,epsfig,amsfonts,amsgen,amsmath,amstext,amsbsy,amsopn,amsthm,cases,listings,color
}
\usepackage{ebezier,eepic}
\usepackage{color}
\usepackage{multirow}
\usepackage{epstopdf}
\usepackage{graphicx}   
\usepackage{pgf,tikz}
\usepackage{mathrsfs}
\usepackage[marginal]{footmisc}
\usepackage{enumitem}
\usepackage[titletoc]{appendix}
\usepackage{booktabs}
\usepackage{url}
\usepackage{mathtools}

\usepackage{pgfplots}
\usepackage{authblk}
\usepackage{amssymb}

\usepackage{wasysym}

\usepackage{empheq}

\usepackage{dsfont}

\usepackage{tikz}
\usepackage{longtable}

\usepackage{subfigure}
\pgfplotsset{compat=1.18}
\usepackage{mathrsfs}
\usepackage{wasysym} 
\usetikzlibrary{arrows}
\usepackage{aligned-overset}
\usepackage{bm}
\usepackage{bbm}
\usepackage[backref=page]{hyperref}

\usepackage{float}

\allowdisplaybreaks[1]

\definecolor{uuuuuu}{rgb}{0.27,0.27,0.27}
\definecolor{sqsqsq}{rgb}{0.1255,0.1255,0.1255}

\newtheorem{definition}{Definition} [section]
\newtheorem{theorem}[definition]{Theorem}
\newtheorem{lemma}[definition]{Lemma}
\newtheorem{proposition}[definition]{Proposition}

\newtheorem{claim}[definition]{Claim}
\newtheorem{problem}[definition]{Problem}

\newtheorem{fact}[definition]{Fact}
\newenvironment{pf}{\noindent{\bf Proof}}

\begin{document}
\title{\bf\Large On a hypergraph Mantel theorem}
\date{\today}
\author[ ]{Xizhi Liu\thanks{Research supported by ERC Advanced Grant 101020255. Email: \texttt{xizhi.liu.ac@gmail.com}}}
\affil[ ]{Mathematics Institute and DIMAP,
            University of Warwick, 
            Coventry, UK}
\maketitle
\begin{abstract}
An $r$-graph is a triangle if there exists a positive integer $i \le \lceil r/2 \rceil$ such that it is isomorphic to the following $r$-graph with three edges: 
\begin{align*}
    \left\{\{1, \ldots, r\},~\{1, \ldots, i, r+1, \ldots, 2r-i\},~\{i+1, \ldots, r, r+1, 2r-i+1, \ldots,2r-1\}\right\}. 
\end{align*}
We prove an Andr{\'a}sfai--Erd\H{o}s--S\'{o}s-type stability theorem for triangle-free $r$-graphs. In particular, it implies that for large $n$, the unique extremal triangle-free construction on $n$ vertices is the balanced complete $r$-partite $r$-graph. The latter result answers a question by Mubayi and Pikhurko~{\cite[Problem~20]{MPS11}} on weakly triangle-free $r$-graphs for large $n$ in a stronger form. 
The proof combines the recently introduced entropic technique of Chao--Yu~\cite{CY24} with the framework developed in~\cite{LMR23unif,HLZ24}. 
\end{abstract}
\section{Introduction}
Given an integer $r\ge 2$, an \textbf{$r$-uniform hypergraph} (henceforth \textbf{$r$-graph}) $\mathcal{H}$ is a collection of $r$-subsets of some finite set $V$.
We identify a hypergraph $\mathcal{H}$ with its edge set and use $V(\mathcal{H})$ to denote its vertex set. 
The size of $V(\mathcal{H})$ is denoted by $v(\mathcal{H})$.
For every vertex $v \in V(\mathcal{H})$, the \textbf{degree} $d_{\mathcal{H}}(v)$ of $v$ in $\mathcal{H}$ is the number of edges containing $v$. 
We use $\delta(\mathcal{H})$, $\Delta(\mathcal{H})$, and $d(\mathcal{H})$ to denote the \textbf{minimum degree}, the \textbf{maximum degree}, and the \textbf{average degree} of $\mathcal{H}$, respectively.  

Given a family $\mathcal{F}$ of $r$-graphs, we say $\mathcal{H}$ is \textbf{$\mathcal{F}$-free}
if it does not contain any member of $\mathcal{F}$ as a subgraph.
The \textbf{Tur\'{a}n number} $\mathrm{ex}(n, \mathcal{F})$ of $\mathcal{F}$ is the maximum number of edges in an $\mathcal{F}$-free $r$-graph on $n$ vertices. 
The \textbf{Tur\'{a}n density} of $\mathcal{F}$ is defined as $\pi(\mathcal{F})\coloneq \lim_{n\to\infty}\mathrm{ex}(n,\mathcal{F})/{n\choose r}$. 
A straightforward averaging argument (see e.g.~\cite{KNS64}) shows that $\mathrm{ex}(n,\mathcal{F})/{n\choose r}$ is non-increasing in $n$, and hence, the limit $\pi(\mathcal{F})$ is well-defined.
We say $\mathcal{F}$ is \textbf{nondegenerate} if $\pi(\mathcal{F}) > 0$.

For $r=2$, the value $\pi(\mathcal{F})$ is well understood thanks to the classical work of Erd\H{o}s--Stone~\cite{ES46} (see also~\cite{ES66}), which extends Tur\'{a}n's seminal theorem from~\cite{TU41}.
For $r \ge 3$, determining $\pi(\mathcal{F})$ is notoriously difficult in general, despite significant effort devoted to this area. 
For results up to~2011, we refer the reader to the excellent survey by Keevash~\cite{Keevash11}.


Given integers $r > i \ge 1$, let $\mathbb{T}_{r,i}$ denote the $r$-graph whose vertex set is $[2r+1]$ and whose edge set is
\begin{align*}
    \left\{\{1, \ldots, r\},~\{1, \ldots, i, r+1, \ldots, 2r-i\},~\{i+1, \ldots, r, r+1, 2r-i+1, \ldots,2r-1\}\right\}. 
\end{align*}
Observe that $\mathbb{T}_{r,i}$ is isomorphic to $\mathbb{T}_{r,r-i}$ for every $i \in [r-1]$.  
Let 
\begin{align*}
    \Delta_{r} \coloneqq \left\{\mathbb{T}_{r,i} \colon 1 \le i \le \lceil r/2 \rceil\right\}.
\end{align*}
For convenience, let $T^{r}(n)$ denote the balanced complete $r$-partite $r$-graph on $[n]$. Prior to Tur\'{a}n's theorem, Mantel~\cite{Mantel07} proved that the maximum size of a $\Delta_{2}$-free graph on $n$ vertices is uniquely achieved by $T^{2}(n)$.
To extend Mantel's theorem to hypergraphs, Katona proposed in the 1960s the problem of determining the value of $\mathrm{ex}(n, \{\mathbb{T}_{3,1}, K_{4}^{3-}\})$, where $K_{4}^{3-}$ denotes the $3$-graph obtained by removing one edge from the complete $4$-vertex $3$-graph $K_{4}^{3}$.
This problem was answered later by Bollob\'{a}s~\cite{BO74}, who proved that the unique extremal construction for $\{\mathbb{T}_{3,1}, K_{4}^{3-}\}$ is $T^{3}(n)$. 
Later, Frankl--F\"{u}redi~\cite{FF83F5} strengthened Bollob\'{a}s's theorem by showing that the same conclusion holds for $\mathbb{T}_{3,1}$ when $n \ge 3000$, thereby establishing the first tight bound for the Tur\'{a}n number of a single hypergraph. 
Their results were further refined in subsequent works, such as~\cite{KM04,Goldwasser,BBHLM16,LM2021,Liu21Cancel,LMR23unif,Liu24canSte,LRW24a,LRW24b}. 

Let $\mathbb{C}_{r}$ denote the collection of $r$-graphs consisting of three edges $A, B, C$ such that the symmetric difference $A \triangle B$ of $A$ and $B$ is contained in $C$. 
Note that $\mathbb{C}_{2} = \Delta_{2} = \{K_3\}$ and $\Delta_{3} \subseteq \mathbb{C}_{3} = \{\mathbb{T}_{3,1}, K_{4}^{3-}\}$.
Motivated by the Mantel Theorem and Bollob\'{a}s's theorem on $\mathrm{ex}(n,\mathbb{C}_{3})$, 
Bollob\'{a}s~\cite{BO74} conjectured that $T^{r}(n)$ is the unique extremal construction for $\mathbb{C}_{r}$ for all $r \ge 4$. 
His conjecture was proved for $r=4$ in a stronger form when $n$ is large by Pikhurko~\cite{PI08} (see~\cite{Sido94Tri} for an asymptotic version). 
However, constructions by Shearer~\cite{She96} show that this conjecture is false for $r \ge 10$. 
Amending Bollob\'{a}s's conjecture, Mubayi and Pikhurko (see~{\cite[Problem~20]{MPS11}}) introduced the weakly triangle-free $r$-graphs and posed the following question: 

An $r$-graph is \textbf{weakly triangle-free} if it does not contain three edges $A, B, C$ such that $C$ contains  strictly more than half of vertices from the symmetric difference $A \triangle B$. 
\begin{problem}[Mubayi--Pikhurko~{\cite[Problem~20]{MPS11}}]\label{PROB:MP}
    Is it true that the maximum size of a weakly triangle-free $r$-graph on $n$ vertices is attained by $T^{r}(n)$?
\end{problem}
Observe that weakly triangle-free $r$-graphs are $\Delta_{r}$-free (but not vice versa in general). 
A very recent result by Chao--Yu~{\cite[Theorem~1.4]{CY24}} shows that for $r \ge 2$, $\pi(\Delta_{r}) = {r!}/{r^{r}}$, the same edge density given by the construction $T^{r}(n)$.
Their ingenious approach employs the concept of the entropic density of hypergraphs that they introduced, which is equivalent to the well-studied Lagrangian of hypergraphs. As a result, similar to proofs using the Lagrangian Method, it does not directly provide much information about the structure of extremal $\Delta_{r}$-free constructions.  
To address this, we combine their approach with the framework (for proving a strong stability of Tur\'{a}n-type problems) established by Mubayi, Reiher, and the author in~\cite{LMR23unif}, along with~{\cite[Theorem~1.1]{HLZ24}}, to show that for large $n$, $T^{r}(n)$ is the unique extremal construction for $\Delta_{r}$. In particular, this answers the question of Mubayi--Pikhurko for large $n$ in a stronger form. 
\begin{theorem}\label{THM:hypergraph-Mantel}
    Let $r\ge 2$ be an integer. 
    There exist $\varepsilon = \varepsilon(r) > 0$ and $N_{0} = N_{0}(r)$ such that the following holds for all $n \ge N_0$. 
    Suppose that $\mathcal{H}$ is a $\Delta_{r}$-free $r$-graph on $n$ vertices with $\delta(\mathcal{H}) \ge n^{r-1}/r^{r-1} - \varepsilon n^{r-1}$. 
    Then $\mathcal{H}$ is $r$-partite. 
    In particular, for large $n$, $T^{r}(n)$ is the unique extremal $\Delta_{r}$-free construction on $n$ vertices. 
\end{theorem}
\textbf{Remarks.}
\begin{itemize}
    \item The constant $\varepsilon(r)$ can be made explicit through a more careful analysis of the proofs, but we did not attempt this, as it is unlikely to yield a tight bound. In fact, determining the optimal value of $\varepsilon(r)$  is an interesting and nontrivial problem. The case $r=2$ was solved by the celebrated Andr{\'a}sfai--Erd\H{o}s--S\'{o}s Theorem~\cite{AES74}, while the case $r=3$ was solved only very recently in~\cite{LRW24a}. 
    \item In Theorem~\ref{THM:hypergraph-Mantel}, the family $\Delta_{r}$ can be replaced by the single $r$-graph (which is not contained in $\Delta_{r}$) described in~{\cite[Theorem~1.5]{CY24}}. The necessary modifications to the proofs in this paper are relatively straightforward, so we omit the details. 
    \item Theorem~\ref{THM:hypergraph-Mantel}, along with the results of Keevash--Lenz--Mubayi~{\cite[Theorem~1.4]{KLM14}} and Kang-- Nikiforov--Yuan~{\cite[Theorem~2]{KNY15}}, solves the $\alpha$-spectral Tur\'{a}n problem for $\Delta_{r}$ (as well as for the $r$-graphs mentioned in the previous remark) for all $\alpha \ge 1$ when $n$ is sufficiently large. For further details, we refer the reader to~\cite{KLM14} and~\cite{HLZ24}. 
\end{itemize}


In the next section, we introduce some necessary definitions and preliminary results. 
In Section~\ref{SEC:uniqueness}, we establish the uniqueness of optimal solution to the Lagrangian of $r$-graphs that contain no homomorphic copy of $\Delta_{r}$ (Proposition~\ref{PROP:unique-solution}).
In Section~\ref{SEC:vtx-extend}, we prove the vertex-extendability of $\Delta_{r}$ with respects to the family of $r$-partite $r$-graphs (Proposition~\ref{PROP:Tr-vtx-extend}). 
In Section~\ref{SEC:proof-hypergraph-mantel}, we present the proof of Theorem~\ref{THM:hypergraph-Mantel}. 
Section~\ref{SEC:remark} contains additional remarks and open problems. 

\section{Preliminaries}\label{SEC:prelim}
Given integers $\ell \ge r \ge 2$, we use $K_{\ell}^{r}$ to denote the complete $r$-graph on $\ell$ vertices. The superscript $r$ will be omitted when $r=2$. 

Let $\mathcal{H}$ be an $r$-graph. 
For every vertex $v \in V(\mathcal{H})$, the \textbf{link} $L_{\mathcal{H}}(v)$ of $v$ in $\mathcal{H}$ is defined as 
\begin{align*}
    L_{\mathcal{H}}(v)
    \coloneqq \left\{e\in \binom{V(\mathcal{H}) \setminus \{v\}}{r-1} \colon e\cup \{v\} \in \mathcal{H}\right\}. 
\end{align*}
For every $i \in [r-1]$, the \textbf{$i$-th shadow} $\partial_{i} \mathcal{H}$ (or simply the shadow $\partial \mathcal{H}$ if $i=1$) is given by 
\begin{align*}
    \partial_{i} \mathcal{H}
    \coloneqq \left\{e\in \binom{V(\mathcal{H})}{r-i} \colon \text{there exists $E\in \mathcal{H}$ such that $e\subseteq E$}\right\}. 
\end{align*}
Denote by $\mathcal{H} - v$ the $r$-graph obtained from $\mathcal{H}$ by removing the vertex $v$ and all edges containing $v$. 

Suppose that the vertex set of $\mathcal{H}$ is $[n]$. Given a collection of disjoint sets $V_1, \ldots, V_{n}$, the \textbf{blowup} $\mathcal{H}[V_1, \ldots, V_n]$ is obtained from $\mathcal{H}$ by replacing each vertex $i$ with the set $V_i$ and each edge with the corresponding complete $r$-partite $r$-graph. 
By \textbf{duplicating a vertex} $v\in \mathcal{H}$, we mean adding a new vertex $\hat{v}$ to $\mathcal{H}$ with the same link as $v$. 

We say $\mathcal{H}$ is \textbf{$2$-covered} if every pair of vertices is contained in at least one edge in $\mathcal{H}$. 
We say $\mathcal{H}$ is \textbf{symmetrized} if it is a blowup of some $2$-covered $r$-graph. 

A family $\mathcal{F}$ of $r$-graphs is \textbf{blowup-invariant} if, for every $\mathcal{F}$-free $r$-graph $\mathcal{H}$, every blowup of $\mathcal{H}$ is $\mathcal{F}$-free. 
Since a blowup of $\mathcal{H}$ can be obtained by duplicating vertices one by one, $\mathcal{F}$ is blowup-invariant iff duplicating any vertex in an $\mathcal{F}$-free $r$-graph $\mathcal{H}$ preserves the $\mathcal{F}$-freeness. 

Given two $r$-graphs $\mathcal{G}$ and $\mathcal{H}$, a map $\psi \colon V(\mathcal{G}) \to \mathcal{H}$ is a \textbf{homomorphism} if $\psi(e) \in \mathcal{H}$ for every $e\in \mathcal{G}$. 
We say $\mathcal{G}$ is \textbf{$\mathcal{H}$-colorable} if there exists a homomorphism from $\mathcal{G}$ to $\mathcal{H}$.

\subsection{Basic properties of triangle-free hypergraphs}
Denote by $\mathcal{T}_{r}$ the collection of $r$-graphs consisting of three edges $A, B, C$ such that 
\begin{align*}
    A \subseteq B \cup C
    \quad\text{and}\quad 
    (B\cap C) \setminus A \neq \emptyset.
\end{align*}
Note that $\Delta_{r} \subseteq \mathcal{T}_{r}$ for $r \ge 2$.
Moreover, an $r$-graph $\mathcal{H}$ is $\mathcal{T}_{r}$-free iff there is no homomorphism from any member of $\Delta_{r}$ to $\mathcal{H}$. 

The following facts are straightforward to verify.
\begin{fact}\label{FACT:Delta-blowup-invariant}
    For every $r \ge 2$, the family $\mathcal{T}_{r}$ is blowup-invariant. 
\end{fact}

An \textbf{$(n,r,r-1)$-system} is an $r$-graph on $n$ vertices such that every set of $r-1$ vertices is contained in at most one edge. 

\begin{fact}\label{FACT:link-triangle-free}
    Let $r \ge 3$ be an integer. 
    Suppose that $\mathcal{H}$ is a $\mathcal{T}_{r}$-free $r$-graph. Then the following statements hold.
    \begin{enumerate}[label=(\roman*)]
        \item\label{FACT:link-triangle-free-1} For every $v\in V(\mathcal{H})$, the link $L_{\mathcal{H}}(v)$ is $\mathcal{T}_{r-1}$-free. 
        \item\label{FACT:link-triangle-free-2} If $\mathcal{H}$ is $2$-covered, then $\mathcal{H}$ is a $\left(v(\mathcal{H}),r,r-1\right)$-system. 
    \end{enumerate}
\end{fact}

\subsection{Vertex-extendability}\label{SUBSEC:vtx-extend}
Let $r\ge 2$ be an integer and $\mathcal{F}$ be a nondegenerate family of $r$-graphs. 
Let $\mathfrak{H}$ be a family of $\mathcal{F}$-free $r$-graphs. We say 
\begin{itemize}
    \item $\mathcal{F}$ is \textbf{edge-stable} with respect to $\mathfrak{H}$ if for every $\delta>0$ there exist $\varepsilon>0$ and $n_0$ such that every $\mathcal{F}$-free $r$-graph $\mathcal{H}$ on $n \ge n_0$ vertices with $|\mathcal{H}| \ge \left(\pi(F)/r! - \varepsilon\right)n^r$ becomes a member in $\mathfrak{H}$ after removing at most $\delta n^r$ edges, 
    \item $\mathcal{F}$ is \textbf{degree-stable} with respect to $\mathfrak{H}$ if there exist $\varepsilon>0$ and $n_0$ such that every $\mathcal{F}$-free $r$-graph $\mathcal{H}$ on $n \ge n_0$ vertices with $\delta(\mathcal{H}) \ge \left(\pi(F)/(r-1)! - \varepsilon\right)n^{r-1}$ is a member in $\mathfrak{H}$, 
    \item $\mathcal{F}$ is \textbf{vertex-extendable} with respect to $\mathfrak{H}$ if there exist $\varepsilon>0$ and $n_0$ such that for every $\mathcal{F}$-free $r$-graph $\mathcal{H}$ on $n \ge n_0$ vertices with $\delta(\mathcal{H}) \ge \left(\pi(F)/(r-1)! - \varepsilon\right)n^{r-1}$ the following holds: 
        if $\mathcal{H}-v$ is a member in $\mathfrak{H}$, then $\mathcal{H}$ is a member in $\mathfrak{H}$ as well. 
    \item $\mathcal{F}$ is \textbf{symmetrized-stable} with respect to $\mathfrak{H}$ if every symmetrized $\mathcal{F}$-free $r$-graph is contained in $\mathfrak{H}$.
\end{itemize}

Vertex-extendability was introduced in~\cite{LMR23unif} to provide a unified framework for proving the degree stability of certain classes of nondegenerate hypergraph families. It was later refined in~\cite{HLZ24}.

\begin{theorem}[{\cite[Theorem~1.7]{LMR23unif}} and {\cite[Theorem~1.1]{HLZ24}}]\label{THM:LMR-vtx-extend}
    Suppose that $\mathcal{F}$ is a blowup-invariant nondegenerate family of $r$-graphs and $\mathfrak{H}$ is a hereditary\footnote{Here, hereditary means that if $H\in \mathfrak{H}$, then every subgraph of $H$ is also contained in $\mathfrak{H}$.} family of $\mathcal{F}$-free $r$-graphs. 
    Then the following statements hold.
    \begin{enumerate}[label=(\roman*)]
        \item\label{THM:LMR-vtx-extend-1} If $\mathcal{F}$ is both symmetrized-stable and vertex-extendable with respect to $\mathfrak{H}$, then $\mathcal{F}$ is degree-stable with respect to $\mathfrak{H}$.
        \item\label{THM:LMR-vtx-extend-2} If $\mathcal{F}$ is both edge-stable and vertex-extendable with respect to $\mathfrak{H}$, then $\mathcal{F}$ is degree-stable with respect to $\mathfrak{H}$.
    \end{enumerate}
\end{theorem}


\subsection{Lagrangian and entropy}\label{SUBSEC:entropy-lagrangian}
First, we introduce the definition of the Lagrangian and some basic properties related to it. This notion was first introduced by Motzkin–Straus~\cite{MS65} for graphs to provide a different proof for the celebrated Tur\'{a}n Theorem, and used later by Frankl--R\"{o}dl~\cite{FR84} for hypergraph Tur\'{a}n problems.

Let $\mathcal{H}$ be an $r$-graph on $[n]$. 
Define the \textbf{Lagrangian polynomial} of $\mathcal{H}$ as 
\begin{align*}
    P_{\mathcal{H}}(X_1, \ldots, X_n)
    \coloneqq \sum_{e\in \mathcal{H}}\prod_{i\in e}X_i.
\end{align*}
For every $i \in [n]$, denote by $\partial_{i}P_{\mathcal{H}}$ the partial derivative of $P_{\mathcal{H}}$ with respect to the $i$-th variable. 
The \textbf{Lagrangian} of $\mathcal{H}$ is defined as 
\begin{align*}
    \lambda(\mathcal{H})
    \coloneqq \max\left\{P_{\mathcal{H}}(x_1, \ldots, x_n) \colon (x_1, \ldots, x_n) \in \mathbb{S}^{n-1}\right\}, 
\end{align*}
where $\mathbb{S}^{n-1}$ is the standard $(n-1)$-dimensional simplex, i.e.
\begin{align*}
    \mathbb{S}^{n-1}
    \coloneqq \max\left\{(x_1, \ldots, x_n) \in \mathbb{R}^{n} \colon x_1 + \cdots + x_n = 1~\text{and}~x_i \ge 0~\text{for}~i \in [n]\right\}.
\end{align*}
For convenience, let 
\begin{align*}
    \mathrm{Opt}(\mathcal{H})
    \coloneqq \left\{(x_1, \ldots, x_n) \in \mathbb{S}^{n} \colon P_{\mathcal{H}}(x_1, \ldots, x_n) = \lambda(\mathcal{H})\right\}.
\end{align*}
Vectors in $\mathrm{Opt}(\mathcal{H})$ are called \textbf{optimal solutions} of $P_{\mathcal{H}}$.

The following fact follows directly from the definition. 
\begin{fact}\label{FACT:Lagrangian}
    Let $\mathcal{G}$ be an $r$-graph on $[m]$ and $\mathcal{H} \coloneqq \mathcal{G}[V_1, \ldots, V_{m}]$ be a blowup of $\mathcal{G}$. 
    Let $n \coloneqq |V_1| + \cdots + |V_m|$ and $x_i \coloneqq |V_i|/n$ for $i \in [m]$. The following statements hold. 
    \begin{enumerate}[label=(\roman*)]
        \item $|\mathcal{H}|/n^r = P_{\mathcal{G}}(x_1, \ldots, x_m)$, and 
        \item for every $j \in [m]$ and every $v \in V_j$, 
        \begin{align*}
            d_{\mathcal{H}}(v)/n^{r-1}
            = \sum_{e\in L_{\mathcal{H}}(j)} \prod_{i \in e} x_i 
            = \partial_{j}P_{\mathcal{H}}(x_1, \ldots, x_m).
        \end{align*}
    \end{enumerate}
\end{fact}

Given a vector $\vec{x} = (x_1, \ldots, x_n) \in \mathbb{R}^{n}$, its \textbf{support} is defined as 
\begin{align*}
    \mathrm{Supp}(\vec{x})
    \coloneqq \left\{i \in [n] \colon x_i \neq 0\right\}.
\end{align*}
The following lemma follows from a standard application of the Lagrangian Multiplier Method (see e.g.~{\cite[Theorem~2.1]{FR84}}).
\begin{lemma}\label{LEMMA:Lagrangian-multiplier}
    Let $\mathcal{H}$ be an $r$-graph on $[n]$. Suppose that $\vec{x} = (x_1, \ldots, x_n) \in \mathrm{Opt}(\mathcal{H})$. 
    Then for every $j \in \mathrm{Supp}(\vec{x})$, it holds that 
    \begin{align*}
        \partial_{j}P_{\mathcal{H}}(x_1, \ldots, x_m)
        = r \cdot \lambda(\mathcal{H}). 
    \end{align*}
\end{lemma}
The following result concerning the Lagrangian 
of $\mathcal{T}_{r}$-free $r$-graphs was established in~\cite{CY24}.
\begin{theorem}[{\cite[Theorem~7.1]{CY24}}]\label{THM:CY24-Lagrangian-density}
    Let $r \ge 2$ be an integer. 
    Suppose that $\mathcal{H}$ is a $\mathcal{T}_{r}$-free $r$-graph. Then $\lambda(\mathcal{H}) = {1}/{r^r}$. 
\end{theorem}

Below, we present definitions and results related to entropy as introduced in~\cite{CY24}.

Let $X$ be a discrete random variable taking values in the set $\Omega$.
For simplicity, let $p_{X}(x) \coloneqq \mathbb{P}[X = x]$ for every $x\in \Omega$. 
The \textbf{support} of $X$ is defined as
\begin{align*}
    \mathrm{Supp}(X)
    \coloneqq \left\{x\in \Omega \colon p_{X}(x) > 0\right\}. 
\end{align*}
Recall that the well-known \textbf{Shannon entropy}~\cite{Sha48} of $X$ is given by  
\begin{align*}
    \mathbb{H}(X)
    \coloneqq - \sum_{x \in \mathrm{Supp}(X)} p_{X}(x) \cdot \log p_{X}(x). 
\end{align*}

Following the definition from~\cite{CY24}, an $r$-tuple of random variables $(X_1, \ldots, X_{r})$ is \textbf{a random edge with uniform ordering} on $\mathcal{H}$ if the following conditions are satisfied:
\begin{enumerate}[label=(\roman*)]
    \item $(X_1, \ldots, X_{r})$ is symmetric, that is, $(X_1, \ldots, X_{r})$ is the same as $(X_{\sigma(1)}, \ldots, X_{\sigma(r)})$ for every permutation $\sigma$ of $[r]$,
    \item $\{X_1, \ldots, X_{r}\}$ is always an edge of $\mathcal{H}$.   
\end{enumerate}
The \textbf{entropy density} $\lambda_{\mathrm{entropy}}(\mathcal{H})$ of $\mathcal{H}$ is defined as 
\begin{align*}
    \lambda_{\mathrm{entropy}}(\mathcal{H})
    \coloneqq \max_{(X_1, \ldots, X_{r})} 2^{\mathbb{H}(X_1, \ldots, X_{r}) - r \cdot \mathbb{H}(X_1)}, 
\end{align*}
where the maximum is taken over all random edges with uniform ordering on $\mathcal{H}$. 

It was shown in~{\cite[Proposition~5.4]{CY24}} that, for every $r$-graph $\mathcal{H}$, 
\begin{align*}
    \lambda_{\mathrm{entropy}}(\mathcal{H}) = r! \cdot \lambda(\mathcal{H}).
\end{align*}

Given an $r$-graph $\mathcal{H}$ on vertex set $V$, denote by $\vec{\mathcal{H}}$ the collection of all ordered $r$-tuples $(i_1, \ldots, i_{r}) \in V^{r}$ such that $\{i_1, \ldots, i_r\} \in \mathcal{H}$. 
Let 
\begin{align*}
    P_{\vec{\mathcal{H}}}(X_1, \ldots, X_n)
    \coloneqq \sum_{(i_1, \ldots, i_{r})\in \vec{\mathcal{H}}} X_{i_1} \cdots X_{i_{r}}
    = r! \cdot P_{\mathcal{H}}(X_1, \ldots, X_n). 
\end{align*}
For every $j \in [r-1]$, the \textbf{$j$-th ordered shadow} $\partial_{j}\vec{\mathcal{H}}$ is defined as 
\begin{align*}
    \partial_{j}\vec{\mathcal{H}}
    \coloneqq \left\{(i_1, \ldots, i_{r-j}) \in V^{r-j} \colon \{i_1, \ldots, i_{r-j}\} \in \partial_{j}\mathcal{H}\right\}.
\end{align*}
For every ordered $j$-tuple $(i_1, \ldots, i_j) \in \partial_{r-j}\vec{\mathcal{H}}$, the \textbf{ordered link} of $(i_1, \ldots, i_j)$ in $\vec{\mathcal{H}}$ is defined as  
\begin{align*}
    L_{\vec{\mathcal{H}}}(i_1, \ldots, i_j)
    \coloneqq \left\{(i_{j+1}, \ldots, i_{r}) \in \partial_{j}\vec{\mathcal{H}} \colon (i_1, \ldots, i_r) \in \vec{\mathcal{H}}\right\}. 
\end{align*}
%

The following result follows from a modification of the proof for~{\cite[Proposition~5.4]{CY24}}.
\begin{proposition}\label{PROP:entropy-difference}
    Let $\mathcal{H}$ be an $r$-graph on $[n]$ and $(x_1, \ldots, x_{n}) \in \mathbb{S}^{n-1}$ be a vector such that $\beta \coloneqq P_{\vec{\mathcal{H}}}(x_1, \ldots, x_n) > 0$. 
    Consider the random edge $(X_1, \ldots, X_{r})$ with uniform ordering on $\mathcal{H}$, given by the following probability distribution: 
    \begin{align}\label{equ:X1-Xr-prob-def}
        \mathbb{P}\left[(X_1, \ldots, X_{r}) = (i_1, \ldots, i_{r})\right]
        \coloneqq 
        y_{i_1, \ldots, i_r}
        \coloneqq \frac{x_{i_1} \cdots x_{i_r}}{\beta} 
        \quad\text{for every}\quad (i_1, \ldots, i_{r})\in \vec{\mathcal{H}}. 
    \end{align}
    Then for every $j \in [r]$, we have 
    \begin{align}\label{equ:entropy-diff-equal}
        & \mathbb{H}(X_1, \ldots, X_{r}) - \frac{r}{j} \cdot \mathbb{H}(X_1, \ldots, X_{j}) 
        = \log \beta 
            - \frac{r}{j} \cdot \sum_{(i_1, \ldots, i_j)} y_{i_1, \ldots, i_{j}} \cdot \log \frac{x_{i_1} \cdots x_{i_j}}{y_{i_1, \ldots, i_{j}}}, 
    \end{align}
    where the summation is taken over $\mathrm{Supp}(X_1, \ldots, X_{j})$, and for each $(i_1, \ldots, i_{j}) \in \partial_{r-j}\vec{\mathcal{H}}$, 
    \begin{align}\label{equ:X1-Xj-prob-def}
        y_{i_1, \ldots, i_{j}}
        \coloneqq \frac{x_{i_1} \cdots x_{i_{j}}}{\beta} \cdot \sum_{(i_{j+1}, \ldots, i_{r}) \in L_{\vec{\mathcal{H}}}(i_1, \ldots, i_{j})} x_{i_{j+1}} \cdots x_{i_{r}}.
    \end{align}
    In particular, 
    \begin{align*}
        \mathbb{H}(X_1, \ldots, X_{r}) - \frac{r}{j} \cdot \mathbb{H}(X_1, \ldots, X_{j})
        \ge \log \beta - \frac{r}{j} \cdot \log P_{\partial_{r-j}\vec{\mathcal{H}}}(x_1, \ldots, x_n).  
    \end{align*}
    Moreover, if $(x_1, \ldots, x_n) \in \mathrm{Opt}(\mathcal{H})$, then 
    \begin{align}\label{equ:entropy-equal-Lagrangian}
        \mathbb{H}(X_1, \ldots, X_{r}) - r \cdot \mathbb{H}(X_1)
        = \log \beta. 
    \end{align}
\end{proposition}
\begin{proof}[Proof of Proposition~\ref{PROP:entropy-difference}]
    Let $(X_1, \ldots, X_{r})$ be the random edge with uniform ordering on $\mathcal{H}$ given by the proposition. 
    Observe that the proposition holds trivially for $j = r$, so we may assume that $j \in [r-1]$. 
    Given the probability distribution in~\eqref{equ:X1-Xr-prob-def}, it is easy to see that for every $(i_1, \ldots, i_{j}) \in \partial_{r-j}\vec{\mathcal{H}}$, 
    \begin{align*}
        \mathbb{P}\left[(X_1, \ldots, X_{j}) = (i_1, \ldots, i_{j})\right]
        = y_{i_1, \ldots, i_{j}}. 
    \end{align*}
    For $j \in [r]$, let $\mathrm{Supp}_{j} = \mathrm{Supp}(X_1, \ldots, X_j)$. 
    Straightforward calculations show that 
    \begin{align*}
        & \sum_{(i_1, \ldots, i_r) \in \mathrm{Supp}_r} \frac{x_{i_1} \cdots x_{i_r}}{\beta} \cdot \log (x_{i_1} \cdots x_{i_{r}}) \\
        & = \frac{1}{\binom{r-1}{j-1}} \cdot \sum_{(i_1, \ldots, i_r) \in \mathrm{Supp}_r} \frac{x_{i_1} \cdots x_{i_r}}{\beta} \cdot \sum_{\{k_1, \ldots, k_j\} \subseteq \{i_1, \ldots, i_r\}}\log (x_{k_1} \cdots x_{k_{j}}) \\
        & = \frac{\binom{r}{j}}{\binom{r-1}{j-1}} \cdot \sum_{(i_1, \ldots, i_j) \in \mathrm{Supp}_j}\frac{x_{i_1} \cdots x_{i_j}}{\beta} \cdot \log (x_{i_1} \cdots x_{i_j})  \cdot \sum_{(i_{j+1}, \ldots, i_r) \in L_{\vec{\mathcal{H}}}(i_1, \ldots, i_j)} x_{i_{j+1}} \cdots x_{i_{r}} \\
        & = \frac{r}{j} \cdot \sum_{(i_1, \ldots, i_j) \in \mathrm{Supp}_j} y_{i_1, \ldots, i_{j}} \cdot \log (x_{i_1} \cdots x_{i_j}). 
    \end{align*}
    It follows that 
    \begin{align*}
        \mathbb{H}(X_1, \ldots, X_{r})
        & = - \sum_{(i_1, \ldots, i_r) \in \mathrm{Supp}_r} y_{i_1, \ldots, i_r} \cdot \log \frac{x_{i_1} \cdots x_{i_r}}{\beta} \\
        & = \sum_{(i_1, \ldots, i_r) \in \mathrm{Supp}_r} y_{i_1, \ldots, i_r} \cdot \log \beta - \sum_{(i_1, \ldots, i_r) \in \mathrm{Supp}_r} \frac{x_{i_1} \cdots x_{i_r}}{\beta} \cdot \log (x_{i_1} \cdots x_{i_{r}}) \\
        & = \log \beta - \frac{r}{j} \cdot \sum_{(i_1, \ldots, i_j) \in \mathrm{Supp}_j} y_{i_1, \ldots, i_{j}} \cdot \log (x_{i_1} \cdots x_{i_j}).
    \end{align*}
    Therefore, we have 
    \begin{align*}
        & \mathbb{H}(X_1, \ldots, X_{r}) - \frac{r}{j} \cdot \mathbb{H}(X_1, \ldots, X_{j}) \\
        & =  \log \beta - \frac{r}{j} \cdot \sum_{(i_1, \ldots, i_j) \in \mathrm{Supp}_j} y_{i_1, \ldots, i_{j}} \cdot \log (x_{i_1} \cdots x_{i_j}) + \frac{r}{j} \cdot \sum_{(i_1, \ldots, i_j) \in \mathrm{Supp}_j} y_{i_1, \ldots, i_{j}} \cdot \log y_{i_1, \ldots, i_{j}} \\
        & = \log \beta - \frac{r}{j} \cdot \sum_{(i_1, \ldots, i_j) \in \mathrm{Supp}_j} y_{i_1, \ldots, i_{j}} \cdot \log \frac{x_{i_1} \cdots x_{i_j}}{y_{i_1, \ldots, i_{j}}},
    \end{align*}
    which proves~\eqref{equ:entropy-diff-equal}. 

    The "In particular`` part follows easily from the following inequality (which follows from Jensen's inequality)
    \begin{align*}
        \sum_{(i_1, \ldots, i_j)} y_{i_1, \ldots, i_{j}} \cdot \log \frac{x_{i_1} \cdots x_{i_j}}{y_{i_1, \ldots, i_{j}}} 
        \le \log \sum_{(i_1, \ldots, i_j)} x_{i_1} \cdots x_{i_j} 
        = \log P_{\partial_{r-j}\vec{\mathcal{H}}}(x_1, \ldots, x_n).
    \end{align*}
    Finally, suppose that $(x_1, \ldots, x_n) \in \mathrm{Opt}(\mathcal{H})$. Then for every $j \in \mathrm{Supp}(\vec{x})$, it follows from Lemma~\ref{LEMMA:Lagrangian-multiplier} that 
    \begin{align*}
        y_{j}
         = \frac{x_j}{\beta} \cdot \sum_{(i_1, \ldots, i_{r-1}) \in L_{\vec{\mathcal{H}}}(j)} x_{i_1} \cdots x_{i_{r-1}} 
        & = \frac{x_j}{r! \cdot \lambda(\mathcal{H})} \cdot (r-1)! \cdot \partial_{j}P_{\mathcal{H}}(x_1, \ldots, x_n)\\
        & = \frac{x_j}{r! \cdot \lambda(\mathcal{H})} \cdot (r-1)! \cdot r \cdot \lambda(\mathcal{H})
        = x_j,
    \end{align*}
    which, combined with~\eqref{equ:entropy-diff-equal}, implies~\eqref{equ:entropy-equal-Lagrangian}. 
    This completes the proof of Proposition~\ref{PROP:entropy-difference}.
%
\end{proof}

\section{Uniqueness of the optimal solution}\label{SEC:uniqueness}
The goal of this section is to establish the following result, which is crucial for the proof of Theorem~\ref{THM:hypergraph-Mantel}.
\begin{proposition}\label{PROP:unique-solution}
    Let $n \ge r \ge 2$ be integers. 
    Suppose that $\mathcal{H}$ is a $2$-covered $\mathcal{T}_{r}$-free $r$-graph on $[n]$. 
    Then $\vec{x} = (x_1, \ldots, x_n) \in \mathrm{Opt}(\mathcal{H})$ iff 
    \begin{enumerate}[label=(\roman*)]
        \item $\mathrm{Supp}(\vec{x})$ has size $r$ and is an edge in $\mathcal{H}$, 
        \item $x_i = 1/r$ for every $i \in \mathrm{Supp}(\vec{x})$. 
    \end{enumerate}
\end{proposition}

Before proving Proposition~\ref{PROP:unique-solution}, let us present several useful lemmas. 

The following result can be derived directly from the proofs Lemma~7.3 and Theorem~7.1 in~\cite{CY24}.
\begin{lemma}[{\cite{CY24}}]\label{LEMMA:stability-inequality-a}
    Suppose that $0 \le x_1 \le \cdots \le x_r = 1$ are real numbers satisfying $x_i + x_j \le x_{i+j}$ for all $1 \le i \le j \le r$. 
    Then $x_1 \cdots x_r \le r!/r^r$, 
    with equality holding iff $x_i = i/r$ for $i \in [r]$. 
\end{lemma}

\begin{lemma}\label{LEMMA:entropy-inequlity}
    Suppose that $\vec{x} = (x_1, \ldots, x_n) \in \mathbb{S}^{n-1}$ satisfies 
    \begin{align}\label{equ:LEMMA:entropy-inequlity-assump}
        \max_{i\in [n]} x_i \le \frac{1}{r}
        \quad\text{and}\quad 
        \sum_{i\in\mathrm{Supp}(\vec{x})} x_i \log x_i = - \log r. 
    \end{align}
    Then, up to a permutation of the indices, we have  
    \begin{align*}
        (x_1, \ldots, x_n)
        = (1/r, \ldots, 1/r, 0, \ldots, 0).
    \end{align*}
\end{lemma}
\begin{proof}[Proof of Lemma~\ref{LEMMA:entropy-inequlity}]
    Suppose to the contrary that this lemma fails. Then there exists $i \in [n]$ such that $0 < x_i < 1/r$.
    By symmetry, we may assume that $i = 1$. 
    Since $\max_{i\in [n]} x_i \le 1/r$ (by~\eqref{equ:LEMMA:entropy-inequlity-assump}), we have 
    \begin{align*}
        \sum_{i\in\mathrm{Supp}(\vec{x})} x_i \log x_i
        & = x_1 \log x_1 + \sum_{i\in\mathrm{Supp}(\vec{x})\setminus \{1\}} x_i \log x_i \\
        & \le x_1 \left(\log x_1 - \log 1/r\right) 
        + x_1 \log 1/r + \sum_{i\in\mathrm{Supp}(\vec{x})\setminus \{1\}} x_i \log 1/r \\
        & = x_1 \left(\log x_1 - \log 1/r\right) +  \log 1/r 
        < \log 1/r,
    \end{align*}
    which contradicts the assumption that $\sum_{i\in\mathrm{Supp}(\vec{x})} x_i \log x_i = - \log r$. 
\end{proof}

We are now ready to prove Proposition~\ref{PROP:unique-solution}. 
\begin{proof}[Proof of Proposition~\ref{PROP:unique-solution}]
    Let $\mathcal{H}$ be a $2$-covered $\mathcal{T}_{r}$-free $r$-graph on $[n]$. 
    Fix $\vec{x} = (x_1, \ldots, x_n) \in \mathrm{Opt}(\mathcal{H})$. 
    It follows from Theorem~\ref{THM:CY24-Lagrangian-density} that 
    \begin{align*}
        \beta
        \coloneqq r! \cdot P_{\mathcal{H}}(x_1, \ldots, x_n)
        = r! \cdot \lambda(\mathcal{H})
        = r!/r^r. 
    \end{align*}
   Additionally, it follows from Lemma~\ref{LEMMA:Lagrangian-multiplier} that for every $j \in \mathrm{Supp}(\vec{x})$, 
    \begin{align}\label{equ:link-Lagrange-Multiplier}
        \sum_{e\in L_{\mathcal{H}}(j)} \prod_{i \in e} x_i 
        =\partial_{j}P_{\mathcal{H}}(x_1, \ldots, x_n)
        = r\cdot  \lambda(\mathcal{H})
        = 1/r^{r-1}. 
    \end{align}
    %
    Let $(X_1, \ldots, X_{r})$ be the random edge  with uniform ordering on $\mathcal{H}$ whose probability distribution is given by~\eqref{equ:X1-Xr-prob-def}. 
    Then for every $j \in [r]$ and for every $(i_1, \ldots, i_j) \in \partial\vec{\mathcal{H}}$, we have 
    \begin{align*}
        \mathbb{P}\left[(X_1, \ldots, X_{j}) = (i_1, \ldots, i_{j})\right]
        = y_{i_1, \ldots, i_{j}},
    \end{align*}
    where $y_{i_1, \ldots, i_{j}}$ is given by~\eqref{equ:X1-Xj-prob-def}. 
    
    Moreover, by Proposition~\ref{PROP:entropy-difference}~\eqref{equ:entropy-equal-Lagrangian}, 
    \begin{align}\label{equ:entropy-X1-Xr-Lagrangian}
        \mathbb{H}(X_1, \ldots, X_r) - r\cdot \mathbb{H}(X_1)
        = \log \beta. 
    \end{align}
    For every $i \in [r]$, define  
    \begin{align*}
        \alpha_i  
        \coloneqq 2^{\mathbb{H}(X_i \mid X_{i+1}, \ldots, X_r) - \mathbb{H}(X_i)}. 
    \end{align*}
    By the chain rule (see e.g.~{\cite[Proposition~3.3]{CY24}}) and the symmetry of $(X_1, \ldots, X_{r})$ that, for every $j \in [r]$, 
    \begin{align*}
        \log(\alpha_r \cdots \alpha_{j})
        & = \sum_{i \in [j, r]} \left(\mathbb{H}(X_i \mid X_{i+1}, \ldots, X_r) - \mathbb{H}(X_i)\right)   \\
        & = \sum_{i \in [j, r]} \left(\mathbb{H}(X_i, X_{i+1}, \ldots, X_r) - \mathbb{H}(X_{i+1}, \ldots, X_r) - \mathbb{H}(X_i) \right)   \\
        & = \mathbb{H}(X_{j}, \ldots, X_r) - (r-j+1) \cdot \mathbb{H}(X_1). 
    \end{align*}
    By symmetry, this can be rewritten as, for every $j \in [r]$, 
    \begin{align}\label{equ:H-X1-Xj}
        \log(\alpha_r \cdots \alpha_{r-j+1})
        = \mathbb{H}(X_{1}, \ldots, X_j) - j \cdot \mathbb{H}(X_1). 
    \end{align}
    \begin{claim}\label{CLAIM:alpha-i-value}
        We have $\alpha_i = i/r$ for $i \in [r]$.
    \end{claim}
    \begin{proof}[Proof of Claim~\ref{CLAIM:alpha-i-value}]
        By definition, we have $\alpha_r = 1$. Moreover, by~{\cite[Lemma~7.2]{CY24}}, for all $i,j \in [r]$ satisfying $i + j \le r$, we have 
        \begin{align}\label{equ:alphai-alphaj}
            \alpha_i + \alpha_j \le \alpha_{i+j} \le \alpha_r =1.
        \end{align}
        Additionally, combining~\eqref{equ:entropy-X1-Xr-Lagrangian} and~\eqref{equ:H-X1-Xj}, we obtain 
        \begin{align*}
            \alpha_1 \cdots \alpha_r
            = \beta
            = r!/r^r,  
        \end{align*}
        which, combined with~\eqref{equ:alphai-alphaj} and Lemma~\ref{LEMMA:stability-inequality-a}, implies that $\alpha_i = i/r$ for $i \in [r]$. 
    \end{proof}

By Claim~\ref{CLAIM:alpha-i-value} and~\eqref{equ:H-X1-Xj}, for every $j \in [r]$, we have 
\begin{align}\label{equ:H-x1-xi}
    \mathbb{H}(X_{1}, \ldots, X_j) - j \cdot \mathbb{H}(X_{1})
    = \log \frac{r(r-1) \cdots (r-j+1)}{r^{j}}.
\end{align}
    \begin{claim}\label{CLAIM:xi-small}
        We have $x_i \le 1/r$ for every $i \in [n]$.
    \end{claim}
    \begin{proof}[Proof of Claim~\ref{CLAIM:xi-small}]
        Suppose to the contrary that $\max_{i\in [n]}x_i > 1/r$. By relabeling the vertices if necessary, we may assume that $x_n > 1/r$. 
        Let $\mathcal{G}$ denote the link of $n$ in $\mathcal{H}$. It follows from Fact~\ref{FACT:link-triangle-free}~\ref{FACT:link-triangle-free-1} that $\mathcal{G}$ is a $\mathcal{T}_{r-1}$-free $(r-1)$-graph on $[n-1]$. 
        For $i \in [n-1]$, let 
        \begin{align*}
            y_i 
            \coloneqq \frac{x_i}{x_1 + \cdots + x_{n-1}}
            = \frac{x_i}{1-x_n}. 
        \end{align*}
        Note that $(y_1, \ldots, y_{n-1}) \in \mathbb{S}^{n-2}$. 
        It follows from Theorem~\ref{THM:CY24-Lagrangian-density} that 
        \begin{align*}
            \frac{1}{(1-x_{n})^{r-1}} \cdot \sum_{e\in L_{\mathcal{H}}(n)} \prod_{i \in e} x_i
            = P_{\mathcal{G}}(y_1, \ldots, y_{n-1})
            \le \lambda(\mathcal{G})
            = \frac{1}{(r-1)^{r-1}}, 
        \end{align*}
        which implies that 
        \begin{align}\label{equ:link-polynomial-upper-bound}
            \sum_{e\in L_{\mathcal{H}}(n)} \prod_{i \in e} x_i
            \le \frac{(1-x_{n})^{r-1}}{(r-1)^{r-1}}
            < \frac{(1-1/r)^{r-1}}{(r-1)^{r-1}}
            = \frac{1}{r^{r-1}}.
        \end{align}
        On the other hand, since $x_n > 0$, it follows from Lemma~\ref{LEMMA:Lagrangian-multiplier} that 
        \begin{align*}
            \sum_{e\in L_{\mathcal{H}}(n)} \prod_{i \in e} x_i
            = r \cdot \lambda(\mathcal{H})
            = \frac{1}{r^{r-1}}, 
        \end{align*}
        which contradicts~\eqref{equ:link-polynomial-upper-bound}.
        This completes the proof of Claim~\ref{CLAIM:xi-small}. 
    \end{proof}
    Since $\mathcal{H}$ is $2$-covered and $\mathcal{T}_{r}$-free, it follows from Fact~\ref{FACT:link-triangle-free}~\ref{FACT:link-triangle-free-2} that $\mathcal{H}$ is an $(n,r,r-1)$-system. This means that for every $\{i_1, \ldots, i_{r-1}\} \in \partial\mathcal{H}$, there exists a unique vertex, denoted by $\psi(i_1, \ldots, i_{r-1})$, such that $\{i_1, \ldots, i_{r-1}, \psi(i_1, \ldots, i_{r-1})\} \in \mathcal{H}$.
    So, by definition~\eqref{equ:X1-Xj-prob-def}, for every $\{i_1, \ldots, i_{r-1}\} \in \mathrm{Supp}(X_1, \ldots, X_{r-1})$, 
    \begin{align}\label{equ:y-1-r-1-explicite}
        y_{i_1, \ldots, i_{r-1}}
        = \frac{x_{i_1} \cdots x_{i_{r-1}}}{\beta} \cdot \sum_{j \in L_{\vec{\mathcal{H}}}(i_1, \ldots, i_{r-1})} x_j
        = \frac{x_{i_1} \cdots x_{i_{r-1}}}{\beta} \cdot x_{\psi(i_1, \ldots, i_{r-1})}. 
    \end{align}
    Additionally, notice that 
    \begin{align}\label{equ:summation-links}
        & \sum_{(i_1, \ldots, i_{r-1}) \in \mathrm{Supp}(X_1, \ldots, X_{r-1})} \frac{x_{i_1} \cdots x_{i_{r-1}}}{\beta} \cdot x_{\psi(i_1, \ldots, i_{r-1})} \cdot \log x_{\psi(i_1, \ldots, i_{r-1})} \notag \\
        & = \frac{1}{\beta} \cdot \sum_{j \in \mathrm{Supp}(X_1)} x_{j} \cdot \log x_j \cdot  \sum_{(i_1, \ldots, i_{r-1}) \in L_{\vec{\mathcal{H}}}(j)} x_{i_1} \cdots x_{i_{r-1}} \notag \\
        & = \frac{(r-1)! \cdot r\cdot \lambda(\mathcal{H})}{\beta} \cdot \sum_{j \in \mathrm{Supp}(X_1)} x_{j} \cdot \log x_j
        = \sum_{j \in \mathrm{Supp}(X_1)} x_{j} \cdot \log x_j, 
    \end{align}
    where the second-to-last equality follows from Lemma~\ref{LEMMA:Lagrangian-multiplier}. 

    Combining~\eqref{equ:y-1-r-1-explicite} and~\eqref{equ:summation-links} with Proposition~\ref{PROP:entropy-difference}~\eqref{equ:entropy-diff-equal} (taking $j = r-1$), we obtain 
    \begin{align*}
        & \mathbb{H}(X_1, \ldots, X_{r}) - \frac{r}{r-1} \cdot \mathbb{H}(X_1, \ldots, X_{r-1}) \\
        & = \log \beta - \frac{r}{r-1} \cdot \sum_{(i_1, \ldots, i_{r-1})} \frac{x_{i_1} \cdots x_{i_{r-1}}}{\beta} \cdot x_{\psi(i_1, \ldots, i_{r-1})} \cdot \log \frac{\beta}{x_{\psi(i_1, \ldots, i_{r-1})}} \\
        & = \log \beta - \frac{r}{r-1} \cdot \sum_{(i_1, \ldots, i_{r-1})} \frac{x_{i_1} \cdots x_{i_{r-1}}}{\beta} \cdot x_{\psi(i_1, \ldots, i_{r-1})} \cdot \left(\log \beta - \log x_{\psi(i_1, \ldots, i_{r-1})} \right) \\
        & = \log \beta - \frac{r}{r-1} \log \beta  - \frac{r}{r-1} \cdot \sum_{(i_1, \ldots, i_{r-1})} \cdot \frac{x_{i_1} \cdots x_{i_{r-1}}}{\beta} x_{\psi(i_1, \ldots, i_{r-1})} \cdot \log x_{\psi(i_1, \ldots, i_{r-1})} \\
        & = - \frac{\log \beta}{r-1} + \frac{r}{r-1} \cdot \sum_{j}x_{j} \log x_j. 
    \end{align*}
    On the other hand, it follows from~\eqref{equ:H-X1-Xj} and Claim~\ref{CLAIM:alpha-i-value} that 
    \begin{align*}
        & \mathbb{H}(X_1, \ldots, X_{r}) - \frac{r}{r-1} \cdot \mathbb{H}(X_1, \ldots, X_{r-1}) \\
        & = \left(\mathbb{H}(X_1, \ldots, X_{r}) - r \cdot \mathbb{H}(X_1)\right) - \frac{r}{r-1} \cdot \left(\mathbb{H}(X_1, \ldots, X_{r-1}) - (r-1) \cdot \mathbb{H}(X_1)\right) \\
        & = \log \beta - \frac{r}{r-1} \cdot \log \frac{r (r-1)\cdots 2}{r^{r-1}} 
         = \log \beta - \frac{r}{r-1} \cdot \log(r \beta)
         = - \frac{\log \beta}{r-1} - \frac{r}{r-1} \log r.
    \end{align*}
    By combining these two inequalities, we obtain 
    \begin{align*}
        \sum_{j}x_{j} \log x_j = -\log r, 
    \end{align*}
    which, together with Claim~\ref{CLAIM:xi-small} and Lemma~\ref{LEMMA:entropy-inequlity}, implies the assertion in Proposition~\ref{PROP:unique-solution}.~\end{proof}

\section{Vertex-extendability of $\mathbb{T}_{r,1}$}\label{SEC:vtx-extend}
In this section, we establish the following result, which is another key ingredient in the proof Theorem~\ref{THM:hypergraph-Mantel}. 
\begin{proposition}\label{PROP:Tr-vtx-extend}
    Let $r \ge 2$ be an integer. Let $\mathfrak{K}_{r}^{r}$ denote the collection of all $r$-partite $r$-graphs. 
    The $r$-graph $\mathbb{T}_{r,1}$ is vertex-extendable with respect to $\mathfrak{K}_{r}^{r}$. 
\end{proposition}
\begin{proof}[Proof of Proposition~\ref{PROP:Tr-vtx-extend}]
    Fix $r\ge 2$. Let $\varepsilon, \varepsilon_1, \varepsilon_2$ be sufficiently small constants such that $0 \le \varepsilon \ll \varepsilon_1 \ll \varepsilon_2$. Let $n$ be a sufficiently large integer. Let $\mathcal{H}$ be an $n$-vertex $\mathbb{T}_{r,1}$-free $r$-graph and $v_{\ast} \in V(\mathcal{H})$ be a vertex such that 
    \begin{enumerate}[label=(\roman*)]
        \item $\delta(\mathcal{H}) \ge \frac{n^{r-1}}{r^{r-1}} - \varepsilon n^{r-1}$, and 
        \item $\mathcal{H} - v_{\ast} \in \mathfrak{K}_{r}^{r}$, this is, $\mathcal{G} \coloneqq \mathcal{H} - v_{\ast}$ is $r$-partite. 
    \end{enumerate}
    We aim to show that $\mathcal{H} \in \mathfrak{K}_{r}^{r}$, this is, $\mathcal{H}$ is $r$-partite as well. 

    Let $V \coloneqq V(\mathcal{H}) \setminus \{v_{\ast}\}$ and let $V_1 \cup \cdots \cup V_{r} = V$ be a partition such that every edge in $\mathcal{H}$ intersects each $V_i$ in exactly one vertex. 
    Since 
    \begin{align*}
        \delta(\mathcal{G})
        \ge \delta(\mathcal{H}) - \binom{n-2}{r-2}
        \ge \frac{n^{r-1}}{r^{r-1}} - 2\varepsilon n^{r-1},
    \end{align*}
    straightforward calculations (see e.g.~{\cite[Corollary~2.3(a)]{LMR23unif}}) show that 
    \begin{align}\label{equ:Vi-size}
        \max_{i \in [r]}|V_i - n/r| \le \varepsilon_1 n.
    \end{align}
    \begin{claim}\label{CLAIM:Vi-independet}
        We have $|e\cap V_i| \le 1$ for every $e\in L_{\mathcal{H}}(v_{\ast})$ and for every $i \in [r]$. 
    \end{claim}
    \begin{proof}[Proof of Claim~\ref{CLAIM:Vi-independet}]
        Suppose to the contrary that there exist $e\in L_{\mathcal{H}}(v_{\ast})$ and $i \in [r]$ such that $|e\cap V_i| \ge 2$. 
        By symmetry, we may assume that $i = 1$. 
        Fix two vertices $u, w \in e\cap V_1$ and let $E_1 \coloneqq e \cup \{v_{\ast}\}$. 
        
        Let $\mathcal{K}$ denote $K_{r-1}^{r-1}[V_2, \ldots, V_{r}]$, the complete $(r-1)$-partite $(r-1)$-graph with parts $V_2, \ldots, V_{r}$. 
        Note that
        \begin{itemize}
            \item  both links $L_{\mathcal{G}}(u)$ and $L_{\mathcal{G}}(w)$ are subgraphs of $\mathcal{K}$, 
            \item $\min\{|L_{\mathcal{G}}(u)|,~|L_{\mathcal{G}}(w)|\} \ge \delta(\mathcal{G}) \ge \frac{n^{r-1}}{r^{r-1}} - 2\varepsilon n^{r-1}$, and 
            \item $|\mathcal{K}| \le \left(\frac{n}{r}  + \varepsilon_1 n\right)^{r-1}$ (which follow from~\eqref{equ:Vi-size}). 
        \end{itemize}
        So it follows from the Inclusion-Exclusion Principle that 
        \begin{align*}
            |L_{\mathcal{G}}(u) \cap L_{\mathcal{G}}(w)|
            \ge 2 \left(\frac{n^{r-1}}{r^{r-1}} - 2\varepsilon n^{r-1}\right) - \left(\frac{n}{r}  + \varepsilon_1 n\right)^{r-1} 
            \ge \frac{n^{r-1}}{r^{r-1}} - \varepsilon_2 n^{r-1}
            > (r-1) n^{r-2}.
        \end{align*}
        So there exists a set $e' \in L_{\mathcal{G}}(u) \cap L_{\mathcal{G}}(w)$ that is disjoint from $e$. Let $E_2 \coloneqq e' \cup \{u\}$ and $E_3 \coloneqq e' \cup \{w\}$. 
        Note that $\{E_1, E_2, E_3\} \subseteq \mathcal{H}$ and $\{E_1, E_2, E_3\}$ is a copy of $\mathbb{T}_{r,1}$, which contradicts the $\mathbb{T}_{r,1}$-freeness of $\mathcal{H}$.  
    \end{proof}
    By Claim~\ref{CLAIM:Vi-independet} and the Pigeonhole Principle, there exists an $(r-1)$-set $\{k_1, \ldots, k_{r-1}\} \subseteq [r]$ such that 
    \begin{align}\label{equ:Pigeonhole-link-v}
        |L_{\mathcal{H}}(v_{\ast}) \cap K^{r-1}_{r-1}[V_{k_1}, \ldots, V_{k_{r-1}}]|
        \ge \frac{|L_{\mathcal{H}}(v_{\ast})|}{r}
        \ge \frac{\delta(\mathcal{H})}{r}
        \ge \frac{n^{r-1}}{r^{r}} - \varepsilon n^{r-1}.
    \end{align}
    
    \begin{claim}\label{CLAIM:link-v-r-1-partite}
        We have $L_{\mathcal{H}}(v_{\ast}) \subseteq K^{r-1}_{r-1}[V_{k_1}, \ldots, V_{k_{r-1}}]$. 
    \end{claim}
    \begin{proof}[Proof of Claim~\ref{CLAIM:link-v-r-1-partite}]
        By symmetry, we may assume that $\{k_1, \ldots, k_{r-1}\} = \{1, \ldots, r-1\}$. 
        Suppose to the contrary that there exists another $(r-1)$-set $\{k'_1, \ldots, k'_{r-1}\} \subseteq [r]$ that is different from $\{1, \ldots, r-1\}$ such that $L_{\mathcal{H}}(v_{\ast}) \cap K^{r-1}_{r-1}[V_{k'_1}, \ldots, V_{k'_{r-1}}] \neq \emptyset$. 
        %
        By symmetry, we may assume that $\{k'_1, \ldots, k'_{r-1}\} = \{2, \ldots, r\}$. 
        
        Fix a set $e_1 \in L_{\mathcal{H}}(v_{\ast}) \cap K^{r-1}_{r-1}[V_{2}, \ldots, V_{r}]$ and let $u$ denote the vertex in $e_1 \cap V_{r}$.  Note that 
        \begin{itemize}
            \item $L_{\mathcal{G}}(u) \subseteq K^{r-1}_{r-1}[V_{1}, \ldots, V_{r-1}]$,
            \item $|L_{\mathcal{G}}(u)| \ge \delta(\mathcal{G}) \ge \frac{n^{r-1}}{r^{r-1}} - 2\varepsilon n^{r-1}$, and 
            \item $|K^{r-1}_{r-1}[V_{1}, \ldots, V_{r-1}]| 
            \le \left(\frac{n}{r} + \varepsilon_1 n\right)^{r-1}$ (which follows from~\eqref{equ:Vi-size}). 
        \end{itemize}
        So, it follows from~\eqref{equ:Pigeonhole-link-v} and the Inclusion-Exclusion Principle that 
        \begin{align*}
            & |L_{\mathcal{H}}(v_{\ast}) \cap L_{\mathcal{G}}(u) \cap K^{r-1}_{r-1}[V_{1}, \ldots, V_{r-1}]|  \\
            & \ge \frac{n^{r-1}}{r^{r}} - \varepsilon n^{r-1} + \frac{n^{r-1}}{r^{r-1}} - 2\varepsilon n^{r-1} - \left(\frac{n}{r} + \varepsilon_1 n\right)^{r-1} 
            > \frac{n^{r-1}}{2r^{r}}
            > (r-1)n^{r-2}. 
        \end{align*}
        Therefore, there exists a set $e_2 \in L_{\mathcal{H}}(v_{\ast}) \cap L_{\mathcal{G}}(u) \cap K^{r-1}[V_{1}, \ldots, V_{r-1}]$ that is disjoint from $e_1$. 
        Let $E_1 \coloneqq e_1 \cup \{v_{\ast}\}$, $E_2 \coloneqq e_2 \cup \{v_{\ast}\}$, and $E_3 \coloneqq e_2 \cup \{u\}$.  Note that $\{E_1, E_2, E_3\} \subseteq \mathcal{H}$ and $\{E_1, E_2, E_3\}$ is a copy of $\mathbb{T}_{r,1}$, which contradicts the $\mathbb{T}_{r,1}$-freeness of $\mathcal{H}$.  
    \end{proof}
    It follows from  Claim~\ref{CLAIM:link-v-r-1-partite} that $\mathcal{H}$ is $r$-partite, this proves that $\mathbb{T}_{r,1}$ is vertex-extendable with respect to $\mathfrak{K}_{r}^{r}$. 
\end{proof}

\section{Proof of Theorem~\ref{THM:hypergraph-Mantel}}\label{SEC:proof-hypergraph-mantel}
In this section, we prove Theorem~\ref{THM:hypergraph-Mantel}. Before that, let us prove the following result. 

\begin{proposition}\label{PROP:min-deg-r-partite}
    Let $r \ge 2$ be an integer.
    There exist $\varepsilon_{\ref{PROP:min-deg-r-partite}} = \varepsilon_{\ref{PROP:min-deg-r-partite}}(r) > 0$ and $N_{\ref{PROP:min-deg-r-partite}} = N_{\ref{PROP:min-deg-r-partite}}(r)$ such that the following holds for all $n \ge N_{\ref{PROP:min-deg-r-partite}}$. 
    Suppose that $\mathcal{H}$ is a symmetrized $\mathcal{T}_{r}$-free $r$-graph on $n$ vertices with $\delta(\mathcal{H}) \ge n^{r-1}/r^{r-1} - \varepsilon_{\ref{PROP:min-deg-r-partite}} n^{r-1}$. Then $\mathcal{H}$ is $r$-partite. 
\end{proposition}
\begin{proof}[Proof of Proposition~\ref{PROP:min-deg-r-partite}]
    Let $\varepsilon > 0$ be sufficiently small and $n$ be sufficiently large. 
    Let $\mathcal{H}$ be a symmetrized $\mathcal{T}_{r}$-free $r$-graph on $n$ vertices with $\delta(\mathcal{H}) \ge n^{r-1}/r^{r-1} - \varepsilon n^{r-1}$.
    By the definition of symmetrized, there exists a $2$-covered $r$-graph $\mathcal{G}$ on $m \coloneqq v(\mathcal{G})$ vertices and a partition $V_1 \cup \cdots \cup V_{m} = V(\mathcal{H})$ such that $\mathcal{H} = \mathcal{G}[V_1, \ldots, V_m]$.
    For convenience, let us assume that the vertex set of $\mathcal{G}$ is $[m]$. 
    Let $x_i \coloneqq |V_i|/n$ for $i \in [m]$. 
    Note from Fact~\ref{FACT:Lagrangian} that 
    \begin{align}\label{equ:H-size-Lagrange}
        P_{\mathcal{G}}(x_1, \ldots, x_m)
        = |\mathcal{H}|/n^{r}
        \ge \frac{n}{r} \cdot \delta(\mathcal{H})/n^r
        \ge 1/r^{r} - \varepsilon/r.
    \end{align}
    Also, for every $j \in [m]$ and every $v\in V_i$, we have 
    \begin{align}\label{equ:degree-Lagrange}
        \sum_{e \in L_{\mathcal{G}}(j)} \prod_{i \in e} x_i
        = \partial_{j} P_{\mathcal{G}}(x_1, \ldots, x_m)
        = d_{\mathcal{H}}(v)/n^{r-1}
        \ge \delta(\mathcal{H})/n^{r-1}
        \ge 1/r^{r-1} - \varepsilon.
    \end{align}

    \begin{claim}\label{CLAIM:m-upper-bound-a}
        We have $m \le \frac{2r^{r-1}}{(r-1)!}$. 
    \end{claim}
    \begin{proof}[Proof of Claim~\ref{CLAIM:m-upper-bound-a}]
        It follows from Fact~\ref{FACT:link-triangle-free}~\ref{FACT:link-triangle-free-2} that $\mathcal{G}$ is an $(m,r,r-1)$-system. This implies that $L_{\mathcal{G}}(i) \cap L_{\mathcal{G}}(j) = \emptyset$ for all distinct $i, j \in [m]$. 
        Combining this with~\eqref{equ:degree-Lagrange}, we obtain 
        \begin{align*}
            m \left(1/r^{r-1} - \varepsilon\right)
            \le \sum_{i\in [m]} \sum_{e \in L_{\mathcal{G}}(i)} \prod_{j \in e} x_j
            \le P_{K_{m}^{r-1}}(x_1, \ldots, x_m)
            \le \binom{m}{r-1}/m^{r-1}
        \end{align*}
        where the last inequality follows from the Maclaurin Inequality. 
        
        It follows that 
        \begin{align*}
            m 
            \le \frac{1}{1/r^{r-1} - \varepsilon} \cdot \binom{m}{r-1}/m^{r-1}
            \le \frac{2r^{r-1}}{(r-1)!},
        \end{align*}
        which proves Claim~\ref{CLAIM:m-upper-bound-a}. 
    \end{proof}
    %
    By Proposition~\ref{PROP:unique-solution}, for every $(y_1, \ldots, y_m) \in \mathbb{S}^{m-1}$, we have  
    \begin{align*}
        P_{\mathcal{G}}(y_1, \ldots, y_m)
        \le r!/r^r, 
    \end{align*}
    with equality holding iff there exists some $r$-set $J \subseteq [m]$ with $J \in \mathcal{G}$ such that $y_i = 1/r$ for every $i \in J$. 
    So, by compactness (see e.g. the proof of~{\cite[Theorem~1.11]{HLZ24}}), there exists $\delta  = \delta(\varepsilon, m) > 0$ such that  
    \begin{align}\label{equ:L-infty-distance}
        \max_{i\in [m]} |x_i - \hat{x}_i| \le \delta,
    \end{align}
    where $\hat{x}_i = 1/r$ iff $i \in J$. 
    Note that we can choose $\varepsilon > 0$ to be sufficiently small at the beginning so that $\delta > 0$ is also sufficiently small. 

    \begin{claim}\label{CLAIM:m-upper-bound-b}
        We have $m  = r$. 
    \end{claim}
    \begin{proof}[Proof of Claim~\ref{CLAIM:m-upper-bound-b}]
        Suppose to the contrary that $m \ge r+1$. 
        Applying~\eqref{equ:degree-Lagrange} to the vertex $m$, we obtain 
        \begin{align*}
            \sum_{e \in L_{\mathcal{G}}(m)} \prod_{i \in e} x_i
            \ge 1/r^{r-1} - \varepsilon
            \ge \frac{1}{2r^{r-1}}. 
        \end{align*}
        Since $\mathcal{G}$ is an $(m,r,r-1)$-system and $J \in \mathcal{G}$, for each term in $\sum_{e \in L_{\mathcal{G}}(m)} \prod_{i \in e} x_i$, the index of at least one $x_i$ comes from the set $[m]\setminus J$. Therefore, by Claim~\ref{CLAIM:m-upper-bound-a} and~\eqref{equ:L-infty-distance}, 
        \begin{align*}
            \sum_{e \in L_{\mathcal{G}}(m)} \prod_{i \in e} x_i
            \le \binom{m-1}{r-1} \cdot \delta 
            \le \binom{2r^{r-1}/(r-1)!}{r-1} \cdot \delta 
            < \frac{1}{2r^{r-1}}, 
        \end{align*}
        contradicting the previous inequality. 
    \end{proof}
    It follows from Claim~\ref{CLAIM:m-upper-bound-b} that $\mathcal{G} \cong K_{r}^{r}$, and hence, $\mathcal{H}$ is $r$-partite. 
    This completes the proof of Proposition~\ref{PROP:min-deg-r-partite}. 
\end{proof}

Next, we present the proof of Theorem~\ref{THM:hypergraph-Mantel}. 
\begin{proof}[Proof of Theorem~\ref{THM:hypergraph-Mantel}]
    Fix $r \ge 2$. Let $\mathfrak{S}$ denote the collection of all $2$-covered $\mathcal{T}_{r}$-free $r$-graphs. 
    Let 
    \begin{align*}
        \mathfrak{H}
        \coloneqq \left\{\mathcal{H} \colon \text{$\mathcal{H}$ is $\mathcal{G}$-colorable for some $\mathcal{G} \in \mathfrak{S}$}\right\}. 
    \end{align*}
    Let us first prove that $\mathcal{T}_{r}$ is degree-stable with respect to $\mathfrak{H}$.
    
    It is clear from the definition that $\mathfrak{H}$ is hereditary and contains all symmetrized $\mathcal{T}_{r}$-free $r$-graphs. 
    By Fact~\ref{FACT:Delta-blowup-invariant}, $\mathcal{T}_{r}$ is blowup-invariant. 
    So, by Theorem~\ref{THM:LMR-vtx-extend}~\ref{THM:LMR-vtx-extend-1}, it remains to show that $\mathcal{T}_{r}$ is vertex-extendable with respect to $\mathfrak{H}$. 

    Let $\varepsilon > 0$ be a sufficiently small constant and $n$ be a sufficiently large integer. 
    Let $\mathcal{H}$ be an $n$-vertex $\mathcal{T}_{r}$-free $r$-graph on $n$ vertices with $\delta(\mathcal{H}) \ge n^{r-1}/r^{r-1} - \varepsilon n^{r-1}$. 
    Suppose that $v_{\ast} \in V(\mathcal{H})$ is a vertex such that $\mathcal{G} \coloneqq \mathcal{H} - v_{\ast}$ is contained in $\mathfrak{H}$. 
    Note that $\mathcal{G} \in \mathfrak{H}$ and 
    \begin{align*}
        \delta(\mathcal{G})
        \ge \delta(\mathcal{H}) - \binom{n-2}{r-2}
        \ge n^{r-1}/r^{r-1} - 2 \varepsilon n^{r-1}. 
    \end{align*}
    So it follows from Proposition~\ref{PROP:min-deg-r-partite} that $\mathcal{G}$ is $r$-partite, that is, $\mathcal{G} \in \mathfrak{K}_{r}^{r}$. 
    By Proposition~\ref{PROP:Tr-vtx-extend}, the $r$-graph $\mathbb{T}_{r,1} \in \mathcal{T}_{r}$ is vertex-extendable with respect to $\mathfrak{K}_{r}^{r}$. Therefore, by the definition of vertex-extendable, $\mathcal{H}$ is contained in $\mathfrak{K}_{r}^{r} \subseteq \mathfrak{H}$ as well. 
    This proves that $\mathcal{T}_{r}$ is vertex-extendable with respect to $\mathfrak{H}$. 

    Notice that the argument above actually shows that $\mathcal{T}_{r}$ is degree-stable with respect to the family $\mathfrak{K}_{r}^{r}$. 
    In particular, $\mathcal{T}_{r}$ is edge-stable with respect to $\mathfrak{K}_{r}^{r}$. 
    Since for every member $F \in \mathcal{T}_{r}$, there exists a member $\tilde{F} \in \Delta_{r}$ such that there exists a homomorphism from $\tilde{F}$ to $F$, a standard application of the Hypergraph Removal Lemma (see e.g.~\cite{RS04,NRS06,Gow07}) shows that every $\Delta_{r}$-free $r$-graph on $n$  vertices can be made $\mathcal{T}_{r}$-free by removing $o(n^r)$ edges. 
    This implies that $\Delta_{r}$ is edge-stable with respect to $\mathfrak{K}_{r}^{r}$ as well. 
    Since, by Proposition~\ref{PROP:Tr-vtx-extend} again, the $r$-graph $\mathbb{T}_{r,1} \in \Delta_{r}$ is vertex-extendable with respect to $\mathfrak{K}_{r}^{r}$, it follows from Theorem~\ref{THM:LMR-vtx-extend}~\ref{THM:LMR-vtx-extend-2} that $\mathbb{T}_{r,1}$ is vertex-extendable with respect to $\mathfrak{K}_{r}^{r}$. This completes the proof of Proposition~\ref{THM:hypergraph-Mantel}. 
\end{proof}

\section{Concluding remarks}\label{SEC:remark}
A standard blowup argument (see e.g. the proof of~{\cite[Theorem~1.2]{LRW24a}}), combined with Theorem~\ref{THM:hypergraph-Mantel}, implies that every $\mathcal{T}_{r}$-free $r$-graph $\mathcal{H}$ on $n \ge r$ vertices satisfies $|\mathcal{H}| \le n^{r}/r^r$, with equality holding iff $r \mid n$ and $\mathcal{H} \cong T^{r}(n)$. 

It seems to be an interesting problem to determine all minimal subfamilies $\mathcal{F}$ of $\Delta_{r}$ such that the extremal $\mathcal{F}$-free construction for large $n$ is still given by $T^{r}(n)$. 
Results by Frankl--F\"{u}redi~\cite{FF89} (see also~\cite{NY17tri}) show that $\mathcal{F}$ cannot be $\{\mathbb{T}_{r,1}\}$ if $r \ge 5$. 

Below is a somewhat interesting application of Theorem~\ref{THM:hypergraph-Mantel} in generalized Tur\'{a}n problems (see e.g.~\cite{Erdos62,AS16} for related background). 

A $3$-graph $\mathbb{S}$ is a \textbf{Steiner triple system} ($\mathrm{STS}$ for short) if every pair of vertices in $\mathbb{S}$ is contained in exactly one edge of $\mathbb{S}$.
It follows from the definition that a $k$-vertex $\mathrm{STS}$ (if exists) contains exactly $\binom{k}{2}/3$ edges. 
A well-known result (see e.g.~\cite{Wil03}) states that a $k$-vertex $\mathrm{STS}$ exists iff $k \in 6\mathbb{N} + \{1,3\}$.

Denote by $\mathrm{ex}(n,\mathbb{S},\mathbb{C}_3)$ the maximum number of copies of $\mathcal{S}$ in an $n$-vertex $\mathbb{C}_3$-free $3$-graphs. 

\begin{theorem}\label{THM:STS-T3}
    Suppose that $\mathbb{S}$ is a Steiner triple system on $k \ge 3$ vertices.
    Then there exists $n_0 = n_0(\mathbb{S})$ such that, for $n \ge n_0$, 
    \begin{align*}
        \mathrm{ex}(n,\mathbb{S},\mathbb{C}_3)
        = |T^{k}(n)|. 
    \end{align*}
\end{theorem}
\begin{proof}[Proof sketch of Theorem~\ref{THM:STS-T3}]
    Let $\mathcal{H}$ be a $\mathbb{C}_3$-free $3$-graph on $n$ vertices that contains the maximum number of copies of $\mathbb{S}$.
    Since $\mathbb{S}$ is $2$-covered and $\mathbb{C}_{3}$ is blowup-invariant, by~{\cite[Proposition~7.1]{CL24}}, we may assume that $\mathcal{H}$ is a blowup of some $2$-covered $3$-graph $\mathcal{S}$. 
    Observe that $\mathcal{S}$ must be an $\mathrm{STS}$ (see Fact~\ref{FACT:link-triangle-free}~\ref{FACT:link-triangle-free-2}). 
    
    Define an auxiliary $k$-graph $\mathcal{G}$ on the same vertex set as $\mathcal{H}$, in which a $k$-set $S \subseteq V(\mathcal{H})$ is an edge in $\mathcal{G}$ iff it spans a copy of $\mathbb{S}$ in $\mathcal{H}$. 
    Note that $\mathcal{G}$ is a blowup of another $k$-graph $\mathcal{G}'$, which arises in the same way from $\mathcal{S}$. 
    
    It can be derived from a result on Steiner triple systems by Doyen--Wilson~{\cite[Section~1]{DW73}} that every pair of edges in $\mathcal{G}'$ share at most $(k-1)/2$ vertices (by showing that the intersection of two copies of $\mathbb{S}$ in $\mathcal{S}$ is also an $\mathrm{STS}$). 
    In particular, $\mathcal{G}'$ is $\mathcal{T}_{k}$-free, and hence, $\mathcal{G}$ is also $\mathcal{T}_{k}$ (by Fact~\ref{FACT:Delta-blowup-invariant}). 
    Then the assertion in Theorem~\ref{THM:STS-T3} follows easily from Theorem~\ref{THM:hypergraph-Mantel}. 
\end{proof}

The proof of Theorem~\ref{THM:hypergraph-Mantel} motivates the following problem on $L$-intersecting families (see~{\cite[Section~9]{FT16}} for a survey). 

Let $r \geq 2$ be an integer and $L \subseteq [r]$ be a subset. An $r$-graph $\mathcal{H}$ is called \textbf{$L$-intersecting} if $|e \cap e'| \in L$ for all distinct edges $e, e' \in \mathcal{H}$.

\begin{problem}\label{PROB:L-intersecting-Lagrangian}
    Determine the maximum value of the Lagrangian of an $L$-intersecting $r$-graph on $n$ vertices. Also, determine the collection of sets $L \subseteq [r]$ such that the maximum value is $1/r^{r}$. 
\end{problem}
\textbf{Remark.} Theorem~\ref{THM:CY24-Lagrangian-density} implies that if $L = [i]$ for some $i < \lceil r/2 \rceil$, then the answer is $1/r^{r}$. 
\section*{Acknowledgements}
The author would like to thank Ting-Wei Chao, Jinghua Deng, Oleg Pikhurko, and Hung-Hsun Hans Yu for related discussions, and Yixiao Zhang for bringing~\cite{CY24} to his attention.
\bibliographystyle{alpha}
\bibliography{HypergraphMantel}
\end{document}